\documentclass[12pt,reqno]{amsart}
\usepackage{amsthm}
\usepackage{amssymb}
\usepackage{graphics}
\usepackage{tikz}
\usetikzlibrary{shapes,backgrounds,calc}
\usepackage{latexsym}
\usepackage{multicol}
\usepackage{verbatim,enumerate}
\usepackage{accents}
\usepackage{cite}
\usepackage{multirow}
\usepackage{bigstrut}
\usepackage{array}
\usepackage[colorlinks=true, linkcolor=blue, citecolor=blue, urlcolor=black]{hyperref}
\usepackage{hyperref}
\usepackage{amsmath, amscd,url}
\usepackage{multirow}
\usepackage{longtable}
\usepackage{stackengine}

\advance\textwidth by 1.3in \advance\oddsidemargin by -.6in \advance\evensidemargin by -.6in
\theoremstyle{definition}

\newtheorem{proposition}{Proposition}
\newtheorem{theorem}{Theorem}[section]
\newtheorem{lemma}[theorem]{Lemma}
\newtheorem{corollary}[theorem]{Corollary}
\theoremstyle{definition}

\newtheorem{example}[theorem]{Example}

\theoremstyle{remark}
\newtheorem{remark}[theorem]{Remark}

\theoremstyle{definition}

\newcounter{cnt}
 \makeatletter
\def\mydggeometry{\makeatletter\dg@YGRID=1\dg@XGRID=20\unitlength=0.003pt\makeatother}
\makeatother \theoremstyle{remark}

\numberwithin{equation}{section}
\let\bwdg\bigwedge
\def\bigwedge{{\textstyle\bwdg}}

\newcommand{\Q}{\mathbb{Q}}
\newcommand{\Z}{\mathbb{Z}}

\newcommand{\nc}{\newcommand}
\newcommand{\rnc}{\renewcommand}

\nc{\cal}{\mathcal} \nc{\goth}{\mathfrak} \rnc{\bold}{\mathbf}

\nc\bomega{{\mbox{\boldmath $\omega$}}} \nc\bpsi{{\mbox{\boldmath $\Psi$}}}
\nc\balpha{{\mbox{\boldmath $\alpha$}}}
\nc\bpi{{\mbox{\boldmath $\pi$}}}
\nc\bvpi{{\mbox{\boldmath $\varpi$}}}
\nc\chara{\operatorname{ch}}

\nc\bxi{{\mbox{\boldmath $\xi$}}}
\nc\bmu{{\mbox{\boldmath $\mu$}}} \nc\bcN{{\mbox{\boldmath $\cal{N}$}}} \nc\bcm{{\mbox{\boldmath $\cal{M}$}}} \nc\blambda{{\mbox{\boldmath
			$\lambda$}}}\nc\bnu{{\mbox{\boldmath $\nu$}}}

\makeatletter
\def\section{\def\@secnumfont{\mdseries}\@startsection{section}{1}%
	\z@{.7\linespacing\@plus\linespacing}{.5\linespacing}%
	{\normalfont\scshape\centering}}
\def\subsection{\def\@secnumfont{\bfseries}\@startsection{subsection}{2}%
	{\parindent}{.5\linespacing\@plus.7\linespacing}{-.5em}%
	{\normalfont\bfseries}}
\makeatother

\nc{\Hom}{\operatorname{Hom}}
\nc{\mode}{\operatorname{mod}}
\nc{\End}{\operatorname{End}} \nc{\wh}[1]{\widehat{#1}} \nc{\Ext}{\operatorname{Ext}} \nc{\ch}{\text{ch}} \nc{\ev}{\operatorname{ev}}
\nc{\Ob}{\operatorname{Ob}} \nc{\soc}{\operatorname{soc}} \nc{\rad}{\operatorname{rad}} \nc{\head}{\operatorname{head}}

\nc{\Cal}{\cal} \nc{\Xp}[1]{X^+(#1)} \nc{\Xm}[1]{X^-(#1)}

\nc{\N}{{\bold N}}  \nc\boa{\bold a} \nc\bob{\bold b} \nc\boc{\bold c} \nc\bod{\bold d} \nc\boe{\bold e} \nc\bof{\bold f} \nc\bog{\bold g}
\nc\boh{\bold h} \nc\boi{\bold i} \nc\boj{\bold j} \nc\bok{\bold k} \nc\bol{\bold l} \nc\bom{\bold m} \nc\bon{\mathbb n} \nc\boo{\bold o}
\nc\bop{\bold p} \nc\boq{\bold q} \nc\bor{\bold r} \nc\bos{\bold s} \nc\boT{\bold t} \nc\boF{\bold F} \nc\bou{\bold u} \nc\bov{\bold v}
\nc\bow{\bold w} \nc\boz{\bold z}\nc\ba{\bold A} \nc\bb{\bold B} \nc\bc{\mathbb C} \nc\bd{\bold D} \nc\be{\bold E} \nc\bg{\bold
	G} \nc\bh{\bold H} \nc\bi{\bold I} \nc\bj{\bold J} \nc\bk{\bold K} \nc\bl{\bold L} \nc\bm{\bold M} \nc\bn{\mathbb N} \nc\bo{\bold O} \nc\bp{\bold
	P} \nc\bq{\bold Q} \nc\br{\bold R} \nc\bs{\bold S} \nc\bt{\bold T} \nc\bu{\bold U} \nc\bv{\bold V} \nc\bw{\bold W} \nc\bz{\mathbb Z} \nc\bx{\bold
	x} \nc\KR{\bold{KR}} \nc\rk{\bold{rk}} \nc\het{\text{ht }}

\nc\toa{\tilde a} \nc\tob{\tilde b} \nc\toc{\tilde c} \nc\tod{\tilde d} \nc\toe{\tilde e} \nc\tof{\tilde f} \nc\tog{\tilde g} \nc\toh{\tilde h}
\nc\toi{\tilde i} \nc\toj{\tilde j} \nc\tok{\tilde k} \nc\tol{\tilde l} \nc\tom{\tilde m} \nc\ton{\tilde n} \nc\too{\tilde o} \nc\toq{\tilde q}
\nc\tor{\tilde r} \nc\tos{\tilde s} \nc\toT{\tilde t} \nc\tou{\tilde u} \nc\tov{\tilde v} \nc\tow{\tilde w} \nc\toz{\tilde z} \nc\woi{w_{\omega_i}}

\begin{document}
	\setcounter{section}{0}
	\setcounter{tocdepth}{1}
	\title{The Index of Composition of polynomials and its applications}
	\author[Surender Kumar]{Surender Kumar}
	\author[Sumandeep Kaur]{Sumandeep Kaur}
	\address[Surender Kumar]{Mathematics Division, School of advanced science and languages, VIT Bhopal University, Kothrikalan, sehore, India}
	\email[Surender Kumar]{surenderkumar@vitbhopal.ac.in}
	\address[Sumandeep Kaur]{Department of Mathematics, Shanghai University, China}
	\email[Sumandeep Kaur]{suman@shu.edu.cn}
	\begin{abstract} 
		The index of a given monic irreducible polynomial $f(x)\in\Z[x]$ having a root $\theta $ is the index of $\Z[\theta]$ in the ring of algebraic integers of $\Q(\theta).$ Further, $f(x)$ is monogenic if its index is $1$. In this article, we prove that the index of $f(x)$ always divides the index of $f(g(x)),$ where $f(x), g(x)\in \Z[x]$. Additionally, we give the precise power of the index of $f(x)$ dividing the index of $f(g(x))$. Also, we provide a necessary condition for the monogenity of $f(g(x))$. As an application of our results, we prove that the $n$-fold composition $f^n(x)=(f\circ f\cdots\circ f)(x)$ of a polynomial $f(x)$ is monogenic for all $n\in\mathbb{N}$ if and only if the sequence $\langle I_n\rangle$ is convergent, where $I_n$ is the index of $f^n(x)$. Moreover, we show that for a given $n\in\mathbb{N}$, if $I_n$ is $m^{\text{th}}$ power free integer, then the $n$-tower $\Q\subseteq K_1 \subseteq\cdots \subseteq K_{n-1} \subseteq K_n,$ defined by $n$ iterates of a polynomial $f(x)\in\Z[x]$ contains at least $\max\{\lfloor n-\log_{\deg f}m \rfloor,0\}$ monogenic number fields.
	\end{abstract}
	
	\maketitle
	\section{Introduction and statement of results}
	An algebraic number field $K$ having ring of integers $\Z_K$ is said to be monogenic if there exists an algebraic integer $\beta\in K$ such that $\Z_K=\Z[\beta]$.  Then also $\Z_K=\Z[\gamma]$ for any $\gamma $ of the form $\pm \beta +a$ with $a\in\Z$. Such elements $\gamma$ are said to be equivalent to $\beta$. These elements are finite up to equivalence \cite{GY}. Monogenic number fields are widely studied in algebraic number theory because of their simple arithmetic structure and computational benefits. Several important families of monogenic number fields arising from binomials, trinomials and quadrinomials have been extensively studied (see\cite{GAAL, AKN, ASS, JS, LJ}).  
	
	An irreducible polynomial $f(x)\in\Z[x]$ having a root $\theta$ is said to be monogenic if the ring of algebraic integers of $\Q(\theta)$ is $\Z[\theta]$. Monogenic polynomial always give a monogenic number field. The Dedekind criterion \cite{DC}, which is based on the coprimality of two polynomials modulo a prime, has been used by numerous mathematicians to determine the monogenity of polynomials \cite{ JS, JON, LJ, SS}. The coprimality condition is quite complex for a general polynomial of higher degree. Despite the difficulty in determining higher degree monogenic polynomials, we are inspired by \cite{GAA, HL, LJones, SSR, SK} concerning the monogenity of the composition of polynomials and monogenity of iterates of a polynomial.
	
	Let $f(x),g(x)\in\mathbb{Z}[x]$ be monic polynomials such that $(f\circ g)(x)$ is irreducible over $\mathbb{Q}$. If $\alpha$ is a root of $(f\circ g)(x)$ and $\beta=g(\alpha)$, then $\beta$ is a root of $f(x)$ and we obtain a tower of number fields
	$\Q\subseteq \Q(\beta)\subseteq \Q(\alpha).$
	In \cite{SK} we showed that any prime divisor of the index $[\Z_K:\Z[g(\alpha)]]$ also divides $[\Z_L:\Z[\alpha]].$ Consequently, this paper addresses the following question: $$ \text{ `` Does the index } [\Z_K:\Z[g(\alpha)]] \text{ divide the index } [\Z_L:\Z[\alpha]]".$$
	\noindent In this article, we study the relation between these indices. In fact, we prove that index of $f(x)$ always divides the index of $f(g(x))$. Additionally, we provide the exact power of the index of $f(x)$ dividing index of $f(g(x))$. Also, we give a sufficient condition for the non-monogenity of $f(g(x))$. As an application of our main results, we give necessary and sufficient condition for all iterates of a polynomial to be monogenic in terms of the convergence of the sequence $\langle I_n\rangle$, where $I_n$ is the index of $f^n(x)$. Moreover, we give a lower bound for the number of monogenic fields in a $n$-tower of fields defined by $n$ iterates of a polynomial.
	
	Throughout the paper, $D(f)$ and $(f\circ g)(x)$ or $f\circ g$ will denote the discriminant of $f(x)\in\Z[x]$ and the composition of $f(x)$ and $g(x)$ respectively. Also the discriminant and the ring of integers of an algebraic number field $K$ will be denoted by $d_K$ and $\Z_K$, respectively. For a relative extension $K/F$, the notations $d_{K/F}$ and $D_{K/F}$ will be used to denote discriminant and different of $K/F$ respectively and $\mathcal{N}_{K/F}(\gamma)$ will denote  the norm of $\gamma\in K$ over $F$. Note that the index to be denoted by $I_f$ of a monic irreducible polynomial $f(x)\in\Z[x]$  having root $\theta$ is the group index $[\Z_{\Q(\theta)}:\Z[\theta]]$. In this paper for the composition, degree of $f(x)$ and $g(x)$ is greater than one and $f(x)\in \Z[x]$ is irreducible means it is irreducible over $\Q$ unless stated otherwise. For  a prime $p$ and an integer $a$, we denote the exact power of $p$ dividing $a$ by $v_{p}(a)$. 
	
	With these notations, we state our main results.
	\begin{theorem}\label{Th1}
		Let $ f(x), g(x)\in\Z[x]$ be two monic polynomials such that $(f\circ g)(x)$ is irreducible. Then $$I^{\deg g}_f ~ \text{ divides }~ I_{f\circ g}.$$
	\end{theorem}
	The following corollary is a direct consequence of the above theorem and Theorem \ref{PTh1}.
	\begin{corollary}\label{C1}
		Let $ f(x), g(x)\in\Z[x]$ be two monic polynomials such that $(f\circ g)(x)$ is irreducible. If $L=\Q(\theta)$ and $K=\Q(g(\theta))$, where $(f\circ g)(\theta)=0,$ then $[\Z_K:\Z[g(\theta)]]^{[L:K]}$ divides  $[\Z_L:\Z[\theta]].$
	\end{corollary}
	\begin{remark}
		In general for any tower $\Q\subseteq K\subseteq L$ of number fields, this result does not hold. For example, if $L=\Q(\sqrt[8]{3})$ and $K=\Q(\alpha)$, where $\alpha=1+2\sqrt[4]{3}+\sqrt{3}$ then $[\Z_L:\Z[\sqrt[8]{3}]]=1$ and $[\Z_K:\Z[\alpha]]=112.$ Clearly $[\Z_K:\Z[\alpha]]\nmid [\Z_L:\Z[\sqrt[8]{3}]].$ 
	\end{remark}
	The following corollary, which follows directly from Theorem~\ref{Th1}, provides a sufficient condition for a prime $p$ not to divide the index of a polynomial.
	\begin{corollary}
		Let $ f(x),g(x)\in\Z[x]$ be two monic polynomials such that $(f\circ g)(x)$ is irreducible. If $p$ is a prime such that $v_{p}(I_{f\circ g})< \deg g$, then $p$ does not divide $I_f.$
	\end{corollary}
	\begin{example}
		Let $f(x)=x^6+x-12$ and $g(x)=x^6-2.$ Then the polynomial $(f\circ g)(x)=x^{36} - 12x^{30} + 60x^{24} - 160x^{18} + 240x^{12} - 191x^6 + 50$ is irreducible. Here, $I_{f\circ g}=2^3\times 5^3.$ Take $L=\Q(\theta)$, where $(f\circ g)(\theta)=0.$ As $v_p(I_{f\circ g})<\deg g$ for any prime $p$, therefore in view of the above corollary $I_f=1.$
	\end{example}
	The next corollary is a well-known result for the families of prime-power compositional polynomials $f(x)$ satisfying the condition that $f(0)$ is square-free (see\cite{SSR}[Theorem 1.11]). However, it is stated here as a consequence of Theorem \ref{Th1}. 
	\begin{corollary}\label{C11}
		Let $p$ be a prime and $t$ be a positive integer. Let $f(x)\in\Z[x]$ be a monic polynomial such that $f(x^{p^t})$ is irreducible. If $f(x)$ is monogenic and $|f(0)|=1,$ then the index of $f(x^{p^t})$ is equal to $p^s,$ for some non-negative integer $s.$
	\end{corollary}

	Let $f(x)\in\Z[x]$ be a monic polynomial such that for each $n\in\mathbb{N}$, its $n$-fold composition $f^n(x)=(f\circ f\circ \cdots\circ f)(x)$ is irreducible. Let  $I_{n}$ denote the index of $f^{n}(x).$ Then the sequence $\langle I_n \rangle$ takes values in $\mathbb{N}.$ 
	We now establish several basic properties of this sequence in the following lemma.
	\begin{theorem}\label{Th3}
		Let $ f(x)\in\Z[x]$ be a monic polynomial of degree $d$ such that $f^n(x)$ is irreducible for every $n\in\mathbb{N}.$ Then the following statements are true:
		\begin{enumerate}
			\item $I_n$ divides  $I_{n+1}$ for all $n\in\mathbb{N}.$
			\item If $m\le n$, then  $I_m^{d^{n-m}}$ divides  $I_n$.
			\item $\langle I_{n}\rangle$ is an increasing sequence.
			\item The iterate $f^n(x)$ is monogenic for all $n\in\mathbb{N}$ if and only if  the sequence $\langle I_{n} \rangle$ is bounded.
			\item If $I_n$ is $d^{\text{th}}$ power free integer, then $I_m=1$ for each $m<n$.
		\end{enumerate}
	\end{theorem}
	It is obvious that if the iterate $f^n(x)=(f\circ f\circ \cdots\circ f) (x)$ is monogenic for all $n\in\mathbb{N}$, then $\langle I_{n} \rangle=1$ for all $n$, i.e., $\langle I_n \rangle$ is convergent. The following proposition establishes that the converse holds as well.
	\begin{corollary}\label{C12}
		Let $ f(x)\in\Z[x]$ be a monic polynomial such that $f^n(x)$ is irreducible for all $n\in\mathbb{N}.$ Then $f^n(x)$ is monogenic for all $n\in\mathbb{N}$ if and only if $\langle I_n \rangle$ is convergent.\end{corollary}
		In view of Theorem \ref{PTh1}, It is  easy to construct a $n$-tower $\Q\subseteq K_1\subseteq\cdots\subseteq K_n$ of number fields defined by $n$ iterates of a polynomial $f(x)\in\Z[x]$ with $[K_i:K_{i-1}]=\deg f$ for each $i=1,2\ldots,n-1.$ Using Theorem \ref{Th1} and Theorem \ref{Th2}, the following corollary gives a formula for a lower bound on the number of monogenic fields in this tower.
		\begin{corollary}\label{P6}
			Let $m,n\in\mathbb{N}$. Let $f(x)\in\Z[x]$ be a monic polynomial of degree $d\ge 2$ such that the $n^{\text{th}}$ iterate $f^n(x)$ is irreducible. Let $\theta_n$ be a root of $f^n(x)$ and let $\theta_i=f(\theta_{i+1})$. If $I_n$ is $m^{\text{th}}$ power free integer, then at least $\max\{\lfloor n-\log_dm \rfloor,0\}$ fields in the following tower are monogenic. $$\Q\subseteq \Q(\theta_1)\subseteq \cdots\subseteq\Q(\theta_{n-1})\subseteq\Q(\theta_n).$$
		\end{corollary}
		\begin{remark} 
			If $m=\deg f$ or $m=2$ in the above corollary, then $ \lfloor n- \log_{\deg f} m \rfloor= n-1$.
				\end{remark}
				Theorem \ref{Th1} shows that if $f$ is non-monogenic, then $f\circ g$ is non-monogenic for every choice of $g$. Moreover, one can easily observe that if either $f$ or $g$ is monogenic, then their composition $f\circ g$ need not be monogenic. For example, let $f(x)=x^3-3$ and $g(x)=x^2-7$. It is clear that $I_f=I_g=1$ but $I_{f\circ g}=4$. So $f\circ g$ is non-monogenic. In the following theorem, we provide a necessary condition involving $I_g$ for the  monogenity of  $f\circ g$.
				\begin{theorem}\label{Th2}
					Let $ f(x),g(x)\in\Z[x]$ be two monic polynomials such that $g(x)$ and $(f\circ g)(x)$ are irreducible. If $(f\circ g)(x)$ is monogenic, then $\gcd\large(I^2_g,f(0)\large)$ is square-free.
				\end{theorem} 
				\begin{example}
				Let $f(x)=x^4+x^3+x^2+x+25$ and $g(x)=x^3-25$, then $I_f=1$ and $I_g=5.$ Note that $f(x)$ is monogenic. Here $\gcd (I_g^2,f(0))=25$, not a square-free integer. Therefore from the above theorem $(f\circ g)(x)$ is non-monogenic.
				\end{example} 
	\section{Preliminary results}
	
 Let $\Q\subseteq K\subseteq L$ be a tower of number fields. Let $\mathfrak{P}$ be a prime ideal of $\Z_L$ and let
	$\mathfrak{p}$ be the prime ideal of $\Z_K$ which lies below $\mathfrak{P}$.
	Then $N_{L/K}(\mathfrak{P}) = \mathfrak{p}^s,$
	with ramification degree $s$, and further 
	for every fractional ideal $I=\displaystyle\prod_{i=1}^{t}\mathfrak{P}^{e_i}_i,$ by multiplicativity, we have
	$N_{L/K}(I) =\displaystyle\prod_{i=1}^t N_{L/K}(\mathfrak{P})^{e_i}.$ We now state some simple properties (see \cite[Chapter-3]{SKJ}, \cite[Chapter-4]{NAR}) of the norm defined above. 
	\begin{lemma}\label{L2.1}
		 If   $I$ and $J$ are non-zero fractional ideals in $L$, then the following statements are true.
		\begin{enumerate}
			\item If $I \subseteq \Z_L$, then $N_{L/K}(I) \subseteq \Z_K$.
			\item If $I \subseteq J$, then $N_{L/K}(I) \subseteq N_{L/K}(J)$.
			\item If $I$ is a principal fractional ideal generated by $a \in L$, then $N_{L/K}(I)$ is the principal fractional ideal of $\Z_K$ generated by $\mathcal{N}_{L/K}(a)$.
			\item The ideal $d_{K/\Q}$ is generated by $d_K$.
			\item If $I \subseteq \Z_K$, then $N_{K/\Q}(I)$ is the principal ideal generated by $N(I)$, where $N(I)$ is the absolute norm of $I$.
			\item   $N_{K/\Q}(I)=N(I)\Z.$
		\end{enumerate}
	\end{lemma}
	\noindent The next two results are standard and can be found in \cite[Chapter-4]{NAR}. 
	\begin{theorem}\label{pth2}
		The different $D_{L/K} $ is generated as an ideal of $\Z_L$ by the set of all differents \[ \delta_{L/K}(\alpha)=
		\begin{cases}
			f'(\alpha), & ~ \text{ if } L=K(\alpha) \text{ with minimal polynomial } f(x)\in\Z_K[x],\\
			0, & ~ \text{ otherwise },
		\end{cases}	\] with $\alpha$ running over $\Z_L.$
	\end{theorem}
	\begin{lemma}\label{L2.3}
		If $\Q\subseteq K\subseteq L$ is a tower of number fields, then $d_{L/\Q}=d_{K/\Q}^{[L:K]}N_{K/\Q}(d_{L/K}).$
	\end{lemma}
	The following lemma will play an important role in the proof of Theorem \ref{Th1}.
		\begin{lemma}\label{L2.5}
		Let $\Q\subseteq K\subseteq L$ be a tower of number fields such that $L$ is generated by $\theta\in\Z_L$ over $K.$ If $g(x)\in\Z_K[x]$ is the minimal polynomial of $\theta,$ then the absolute norm of the ideal  $D_{L/K}$  divides the absolute value of $\mathcal{N}_{L/\Q}(g'(\alpha)).$ 
	\end{lemma}
	\begin{proof} 
		Here $L=K(\theta)$ and $g(x)$ is the minimal polynomial of $\theta\in\Z_L$ over $K.$ From Theorem \ref{pth2}, we have 
		$D_{L/K}= \langle\{ \delta_{L/K}(\alpha) ~ ~ | ~~ \alpha\in\Z_L\}\rangle.$ This shows that
		$g'(\theta)\in D_{L/K}$ and therefore $ 
		g'(\theta)\mathbb{Z}_L\subseteq D_{L/K}.$ Using $(2)$ of Lemma \ref{L2.1}, we see that \begin{equation}\label{EQa}
			N_{L/K}(g'(\theta)\Z_L)\subseteq N_{L/K}(D_{L/K}).
		\end{equation} Using $(3)$ of Lemma \ref{L2.1}, we get $N_{L/K}(g'(\theta)\Z_L)=\mathcal{N}_{L/K}(g'(\theta))\Z_K$. By $\eqref{EQa}$ and in view of the fact that $\Z_K$ is a Dedekind domain,  we can write 
		\begin{equation}\label{EQb}
			\mathcal{N}_{L/K}(g'(\theta))\Z_K=N_{L/K}(D_{L/K})C ,\end{equation} for some ideal $C$ of $\Z_K.$ Apply norm $N_{K/\Q}$ on both side of \eqref{EQb}, we obtain 
		$$N_{K/\Q}(\mathcal{N}_{L/K}(g'(\theta))\Z_K))=N_{K/\Q}(N_{L/K}(D_{L/K}))N_{K/\Q}(C).$$ Hence by $(3)$ of Lemma \ref{L2.1} and multiplicative property of $\mathcal{N}$, we see that $N_{L/\Q}(D_{L/K})$ divides  
		$\mathcal{N}_{L/\mathbb{Q}}(g'(\theta))\Z.$ Write $\mathcal{N}_{L/\Q}(g'(\theta))\Z=N_{L/\mathbb{Q}}(D_{L/K})A,$ for some ideal $A\subseteq \Z.$ Then by $(6)$ of Lemma \ref{L2.1}, $\mathcal{N}_{L/\Q}(g'(\theta))\Z=N(D_{L/K}))\Z A,$ where $N$ denotes the absolute norm. Using multiplicative property of $N$,  we get$|\mathcal{N}_{L/\Q}(g'(\theta))|=N(D_{L/K})N(A).$ 
		This implies that $N(D_{L/K})$ divides $|\mathcal{N}_{L/\Q}(g'(\alpha))|.$
	\end{proof}
	The following result is well-known \cite{SK} and will be used in the proof section and last section. 
	\begin{lemma}\label{lemmaB}
		Let $f(x)$ and $g(x)$ be two non-zero monic polynomials belonging to $\Z[x]$ such that $(f \circ g)(x)$ is irreducible, then the discriminant formula of  $(f\circ g)(x)$ is given by $$D(f\circ g)=\pm D(f)^{\deg g}\mathcal{N}_{\Q(\alpha)/\Q}(g'(\alpha)),$$ where $\alpha$ is a root of $(f\circ g)(x).$ 
	\end{lemma}
	\begin{remark}\label{remark}
		Note that if $g(x)=x^{n}$ in the above lemma, then the discriminant formula of $f(x^n)$ is  $$D(f(x^n))=\pm D(f)^{n}n^{n\deg f}f(0)^{n-1}.$$
	\end{remark}
	 To investigate the monogenity of $f\circ g$ in Theorem \ref{Th2}, we will follow the approach of K. Uchida, so we state the following two results which are proved in \cite{UC}.
	
\begin{lemma}\label{UL}	An  ideal $\mathfrak M$ of $\mathbb{Z}[x]$ containing a monic polynomial is maximal if and only if $\mathfrak M =  \langle p, g(x) \rangle $  for some prime number   $p$  and a monic polynomial  $g(x)$ belonging to $\mathbb Z [x]$ which is irreducible modulo $p$.
\end{lemma}	
	
	\begin{theorem}\label{UT}
		Let  $K=\mathbb Q(\theta)$  with  $\theta$ in $\Z_K$ having  minimal polynomial $f(x)$ over $\Q$, then $\Z_K=\Z[\theta]$ if and only if $f(x)$ does not belong to $\mathfrak M ^2$ for any maximal ideal $\mathfrak M$ of the polynomial ring $\mathbb Z[x]$.
	\end{theorem}
	We now prove the following result which is useful in the proof of Theorem \ref{Th2}.
	
	\begin{lemma}\label{PL1}
		Let $f(x), g(x)\in\Z[x]$ be two monic polynomials such that  $f\circ g$ is irreducible. Let $h(x)$ be a monic irreducible polynomial modulo a prime $p$ such that $g(x)\in \langle p, h(x) \rangle^2$. Then $$f(g(x))\in \langle p,h(x)\rangle^2 ~ ~ ~ \iff ~ ~ ~ p^2 \text{ divide } f(0).$$
	\end{lemma}
	\begin{proof}
		Let $f(x)=x^n+a_{n-1}x^{n-1}+\cdots+a_1x+a_0.$ Then 
			$f(g(x))=g(x)b(x)+a_0,$ where $b(x)=g(x)^{n-1}+a_{n-1}g(x)^{n-2}+\cdots+a_2g(x)+a_1.$ As $h(x)$ is a monic irreducible polynomial modulo $p$ and $g(x)\in \langle p,h(x)\rangle^2.$ Therefore $g(x)b(x)\in \langle p,h(x)\rangle^2.$ This shows that $f(g(x))\in \langle p,h(x)\rangle^2$ if and only if  $a_0\in \langle p,h(x)\rangle^2.$
		Equivalently, 
		$f(g(x))\in \langle p,h(x)\rangle^2$ if and only if $a_0=p^2u(x)+ph(x)v(x)+h^2(x)w(x),$ for some $u(x),v(x), w(x)\in\Z[x].$ This happens only when $p^2$ divides $ f(0)$. \end{proof}
	
	Let $F$ be a field and let $f(x),g(x) \in F[x]$. The following theorem, due to Capelli \cite[Page 288]{NG}, gives a necessary and sufficient condition for the composition $f \circ g$ to be irreducible over $F$.
	\begin{theorem}\label{PTh1}
		Let $F$ be a field. Let $f(x)$ and $g(x)$ be two monic polynomials belonging to $F[x]$ and $\beta$ be a root of $f(x)$ in the algebraic closure of $F$, then $(f\circ g)(x)$ is irreducible over $F$ if and only if 
		 $f(x)$ is irreducible over $F$ and $g(x)-\beta$ is irreducible over $F(\beta).$
	\end{theorem}
	\section{Proof of Theorem \ref{Th1}, Corollary \ref{C11}, Theorem \ref{Th3}, Corollaries \ref{C12}, \ref{P6} and Theorem \ref{Th2}.}
	\begin{proof}[Proof of Theorem \ref{Th1}] Set 
		$K=\mathbb{Q}(g(\theta))$,  
		$L=\mathbb{Q}(\theta),$ where
		$(f\circ g)(\theta)=0.$ From Theorem \ref{PTh1}, it is clear that $h(x)=g(x)-g(\theta)\in\Z_K[x]$ is irreducible over $\Z_K$ and $h(\theta)=0$. Using Lemmas~\ref{L2.3} and~\ref{lemmaB} we have 
		$$D(f\circ g)=\pm D(f)^{[L:K]}\, \mathcal{N}_{L/\mathbb{Q}}\!\big(g'(\theta)\big) ~ ~ \text{ and } ~ ~ ~
			d_{L/\Q}= d_{K/\Q}^{[L:K]}\, N_{K/\mathbb{Q}}(d_{L|K}).$$
		\noindent From $(4)$  and $(5)$ of Lemma \ref{L2.1}, we get \begin{equation}\label{eqr}
		d_{L}= \pm d_{K}^{[L:K]}\, N(N_{K/\mathbb{Q}}(d_{L/K})),\end{equation} where $N(N_{K/\mathbb{Q}}(d_{L/K}))$ is the absolute norm of $N_{K/\mathbb{Q}}(d_{L/K}).$ Now by well known identities $D(f)=I_f^2d_K$ and $D(f\circ g)=I_{f\circ g}^2d_L$, we see that
		$I_{f \circ g}^2\, d_L
			= \pm D(f)^{[L:K]}\, \mathcal{N}_{L/\mathbb{Q}}\!\big(g'(\theta)\big)$
		Using \eqref{eqr}, we obtain
		$$I_{f\circ g}^2\, d_K^{[L:K]}\, N(N_{K/\mathbb{Q}}(d_{L|K}))= (I_f^2\, d_K)^{[L:K]}\, \mathcal{N}_{L/\mathbb{Q}}\!\big(g'(\theta)\big).$$ By multiplicative property of norm and $(6)$ of lemma \ref{L2.1}, we get
		$\frac{I_{f\circ g}^2}{I_f^{2[L:K]}}=\frac{
				|\mathcal{N}_{L/\mathbb{Q}}(g'(\theta))|
			}{
				N(D_{L/K})
			}.$
		 Since $g'(\theta)\in D_{L/K}$, we see that $N_{L/\mathbb{Q}}(g'(\theta))\in\Z.$ Hence the conclusion follows from Lemma \ref{L2.5}.
	\end{proof}
	\begin{proof}[Proof of Corollary \ref{C11}]
		$(1)$ Let $\theta$ be a root of $F(x):=f(x^{p^t})$ and let $\theta_0=\theta^{p^t}$.  So by Theorem \ref{PTh1}, $[\Q(\theta):\Q(\theta_0)]=p^{t}$. Using Remark \ref{remark} with the well-known identity $D(F)=I_F^2d_{\Q(\theta)}$, we see that \begin{equation*}\label{eqP1}
			D(F(x))=(d_{\Q(\theta_0)})^{p^t}p^{tp^t\deg f}=I_F^2d_{\Q(\theta)}.\end{equation*}
		 As $d_{\Q(\theta_0)}^{p^t}$ divides $d_{\Q(\theta)},$ therefore $I_F=p^s$, for some non-negative integer $s.$
		\end{proof}

	\begin{proof}[Proof of Theorem \ref{Th3}]
		$(1)$ Here $f^{n+1}(x)=(f\circ f\circ \cdots\circ f)\circ f(x)=f^n(f(x)).$ Let $L=\Q(\theta)$ and $K=\Q(f^n(\theta)),$ where $\theta$ is a root of $f^{n+1}(x).$ Clearly from Theorem \ref{PTh1}, $f(x)-f(\theta)$ is irreducible over $K$ and so $[L:K]=\deg f=d.$ By Theorem \ref{Th1}, we conclude that $I_{f^n\circ f}$ is divisible by $I_f^d.$ \\
		$(2)$ From the above case, we see that $I^d_n$ divides $I_{n+1}$ for all $n\in \mathbb{N}.$ Let $m\le n.$ Then the degree of $f^m(x)$ and $f^n(x)$ is $d^m$ and $d^n$ respectively. Set $n=m+t.$ Let $K_{m+i}=\Q(f^{t-i}(\theta))$ for $0\le i\le t-1$ and $L=\Q(\theta)$,  where $\theta$ is a root of $f^n(x)$.  Clearly we have the following tower of number fields. $$K_m\subseteq K_{m+1}\subseteq\cdots\subseteq K_{n-1}\subseteq L.$$ From Theorem \ref{PTh1}, $[K_{m+i+1}:K_{m+i}]=\deg f=d$. A simple induction shows that $I_m^{d^t}$ divides $I_n.$\\
		$(3)$ Clearly from the above case, the sequence $\langle I_n\rangle$ is an increasing sequence.\\
		$(4)$ First suppose that $f^n(x)$ is monogenic for all $n\in\mathbb{N}$.  Then $I_n=1$ for all $n\in \mathbb{N}.$ So $I_n$ is bounded.  Conversely, let there exist a positive real number $M$ such that $I_n\le M$  for all $n\in\mathbb{N}.$ Suppose for some $m\in \mathbb{N},$ the $m$-fold composition $f^m(x)$ of $f(x)$ is non-monogenic, then $I_m>1$ and $\deg f^m(x)=d^m.$ Since $M$ is fix, we can easily choose a natural number $t$ such that $I_m^t>M.$ For a suitable natural number $u>t$ we obtain $I_m^{d^u}>I_m^t.$ By Case (2), we get $I_m^{d^u}\mid I_{m+u}$ and so $I_{m+u}>M$, a contradiction. Hence for all $n,$ the $n^{\text{th}}$ iterate $f^n(x)$ is monogenic.\\
		$(5)$ Here $I_n$ is $d^{\text{th}}$ power free integer. Suppose for some $m<n,$ $I_m>1.$ Then using Case $(2)$, $I_m^{d^{n-m}}\mid I_n$, which is not possible. So $I_m=1$ for all $m<n.$
	\end{proof}
	\begin{proof}[Proof of Corollary \ref{C12}]
		First suppose that $f^n(x)$ is monogenic for all $n\in\mathbb{N}.$ Then $I_n=1$ for all $n.$ Therefore the sequence $\langle I_n \rangle$ is convergent. Conversely, suppose $\langle I_n \rangle$ is convergent. Since $I_n$ is a sequence of natural number so it must be eventually constant, i.e., there exist some natural numbers $c$ and $N$ such that $I_n=c ~ ~  \text{ for all } n\ge N.$ Now using $(2)$ of Theorem \ref{Th3}, $I_{N}^{\deg f}$ divides $I_{N+1}.$ That is, $c^{\deg f}$ divides $c.$ This happens only when $c=1.$ Thus $I_n=1$ for all $n\ge N.$ Next, using $(5)$ of Theorem \ref{Th3}, we see that $I_n=1$ for all $n<N.$ Hence for all $n\in\mathbb{N},$ the $n^{\text{th}}$ iterate $f^n(x)$ is monogenic.
	\end{proof}
	\begin{proof}[Proof of Corollary \ref{P6}]
		Let $k\le n$. Clearly from $(2)$ of Theorem \ref{Th3} we have  $I_k^{d^{n-k}}\mid I_n.$ For $I_k=1,$ we let $d^{n-k} \ge m$. Then we have $d^{n}\ge md^k>1.$ This implies that $n\log d\ge \log m +k\log d.$ So we get $n\log d- \log m \ge k\log d,$ i.e., $n- \log_d m \ge k.$ This means that $k= \lfloor n- \log_d m \rfloor,$ where $\lfloor a \rfloor$ denotes the greatest integer less than or equal to $a.$
	\end{proof}
	\begin{proof}[Proof of Theorem \ref{Th2}]
		 Suppose that $\gcd(I_g^2,f(0))$ is not square-free. Let $p$ be a prime dividing $\gcd(I_g^2,f(0))$, then by Lemma \ref{UL} and Theorem \ref{UT}, $g(x)\in \langle p, h(x)\rangle^2$ for some monic irreducible factor $h(x)$ of $g(x)$ modulo $p.$ From Lemma \ref{PL1}, we known that $f(g(x))\in \langle p, h(x)\rangle^2$ if and only if  $p^2\mid f(0).$ As $p^2\mid f(0)$, so $h(x)$ is also an irreducible factor of $f(g(x))$ modulo $p$. Thus by Lemma \ref{UL} and Theorem \ref{UT}, we conclude that $p\mid I_{f\circ g}.$
	\end{proof}
	\section{Infinite examples}
	In this section, we present some direct consequences of our main results that are already known. The first case concerns the family of power compositional polynomials. Here, our general theorems yield a slightly weaker version of \cite[Theorem 1.1]{SSR}.
	\begin{proposition}
		Let $k\ge 2$ be a positive integer. Let $f(x)\in\Z[x]$ be a monic polynomial such that $f(x^k)$ is irreducible. If $f(x^k)$ is monogenic, then \begin{enumerate}
			\item $f(x)$ is monogenic,
			\item $p$ does not divide the index of $f(x^{p})$ for any prime $p\mid k$, and 
			\item $f(0)$ is square-free.
		\end{enumerate}
	\end{proposition}
	\begin{proof}
		First suppose that $F(x)$ is monogenic. Using Theorem \ref{Th1}, we see that each $f(x^p)$  is monogenic and $f(x)$ is also monogenic. Using Lemma \ref{PL1} together with Theorem \ref{UT}, we can easily observe that $f(0)$ must be square-free.
	\end{proof}
	\begin{proposition}
		Let $k\ge 2$ be a positive integer. Let $f(x)\in\Z[x]$ be a monic polynomial such that $f(x^k)$ is irreducible. If  \begin{enumerate}
			\item $f(x)$ is monogenic,
			\item $p$ does not divide the index of $f(x^{p^{v_p(k)}})$ for any prime $p\mid k$, and 
			\item $|f(0)|=1$,
		\end{enumerate} then $f(x^k)$ is monogenic.
	\end{proposition}
	\begin{proof}
		 Write $k=p_1^{e_1}p_2^{e_2}\cdots p_t^{e_t}$ as a product of distinct  prime powers. Let $k_i=\frac{k}{p_i^{e_i}}$ for $1\le i\le t$. Set $F_i(x)=f(x^{p_i^{e_i}})$. Then $F(x):=f(x^k)=F_i(x^{k_i})$ for all $1\le i\le t.$ Clearly from Theorem \ref{PTh1}, each $F_i$ is irreducible. Suppose that $\theta$ is a root of $F(x)$, then $F_i(\theta^{k_i})=0$ and $f(\theta^k)=0.$ Let $\theta_i=\theta^{k_i}$, then $[\Q(\theta):\Q(\theta_i)]=k_i$ for all $i=1,2\ldots,t.$
		
		In view of Remark \ref{remark} and the fact that $D(F(x))=I_F^2d_{\Q(\theta)}$, we have \begin{align*}I_F^2d_{\Q(\theta)}&=\pm D(F_i)^{k_i}k_i^{k\deg f},\\
			&=\pm I^{2k_i}_{F_i}(d_{\Q(\theta_i)})^{k_i}k_i^{k\deg f}.	\end{align*} 
		Clearly Theorem \ref{Th1}, implies that $I_{F_i}^{2k_i}$ divides $I_F^2$. Also we know that  $d_{\Q(\theta_i)}^{k_i}$ divides $d_{\Q(\theta)}$.
		Using Corollary \ref{C11} and given condition (2), for each $i$ we have $I_{F_i}=1$. Also, for each $1\le i\le t,$ $k_i=\frac{k}{p_i^{e_i}}$ implies that $p_i$ does not divide $I_F.$ Thus any prime divisor of $k$ will not divide $I_F.$ Therefore the identity $I^2_Fd_{\Q(\theta)}=\pm I^{2k}_{f}(d_{\Q(\theta_k)})^{k}k^{k\deg f}$ implies that $I_F=1.$
	\end{proof}

	\noindent\textbf{Statements and Declarations}
	
	\noindent\textbf{Funds}
	The authors declare that no funds, grants, or other support were received during the preparation of this manuscript.
	
	\noindent\textbf{Competing Interests} The authors have no relevant financial or non-financial interests to disclose.
	
	\noindent\textbf{Author Contributions} All authors contributed equally to this work. All authors have read and approved the final manuscript.
	
	\noindent\textbf{Data Availability} No data were generated or analyzed during the current study.
\end{document}